\documentclass[a4paper,12pt,intlimits,oneside]{amsart}

\usepackage{enumerate}
\usepackage{amsfonts,amsmath}

\usepackage{latexsym,amssymb}

\usepackage{mathrsfs}

\newcommand{\comment}[1]{}

\newcommand{\R}{{\mathbb R}}
\newcommand{\X}{{\mathcal X}}

\def\M{{\mathcal M}}

\begin{document}

\title[On weak$^*$-convergence in $H^1_\rho(\X)$]{On weak$^*$-convergence in the localized Hardy spaces $H^1_\rho(\X)$ and its application}         % Enter your title between curly braces

\author{Dinh Thanh Duc, Ha Duy Hung and Luong Dang Ky$\,^*$}


\thanks{The first author is supported by Vietnam National Foundation for Science and Technology Development (Grant No. 101.02-2014.32)\\	
	$^{*}$Corresponding author}

\begin{abstract}
Let $(\mathcal X, d, \mu)$ be a complete RD-space. Let $\rho$ be an admissible function on $\X$, which means that $\rho$ is a positive function on $\X$ and there exist positive constants $C_0$ and $k_0$ such that, for any $x,y\in \X$,
$$\rho(y)\leq C_0 [\rho(x)]^{1/(1+k_0)} [\rho(x)+d(x,y)]^{k_0/(1+k_0)}.$$
 In this paper, we define a space $VMO_\rho(\X)$ and show that it is the predual of the localized Hardy space $H^1_\rho(\X)$ introduced by Yang and Zhou \cite{YZ}. Then we prove a version of the classical theorem of Jones and Journ\'e \cite{JJ} on weak$^*$-convergence in $H^1_\rho(\X)$. As an application, we give an atomic characterization of $H^1_\rho(\X)$.
\end{abstract}

\maketitle
\newtheorem{theorem}{Theorem}[section]
\newtheorem{lemma}{Lemma}[section]
\newtheorem{proposition}{Proposition}[section]
\newtheorem{remark}{Remark}[section]
\newtheorem{corollary}{Corollary}[section]
\newtheorem{definition}{Definition}[section]
\newtheorem{example}{Example}[section]
\numberwithin{equation}{section}
\newtheorem{Theorem}{Theorem}[section]
\newtheorem{Lemma}{Lemma}[section]
\newtheorem{Proposition}{Proposition}[section]
\newtheorem{Remark}{Remark}[section]
\newtheorem{Corollary}{Corollary}[section]
\newtheorem{Definition}{Definition}[section]
\newtheorem{Example}{Example}[section]
\newtheorem*{theoremcw}{Theorem C-W}
\newtheorem*{theorema}{Theorem A}
\newtheorem*{theoremb}{Theorem B}
\newtheorem*{theoremc}{Theorem C}

\section{Introduction}

It is a well-known and classical result (see \cite{CW}) that the space $BMO(\R^n)$ is the dual of the Hardy space $H^1(\R^n)$ one of the few examples of separable, nonreflexive Banach space which is a dual space. In fact,  let $C_c(\R^n)$ be the space of all continuous functions with compact support and denote by $VMO(\mathbb R^n)$ the closure of $C_c(\R^n)$ in $BMO(\mathbb R^n)$, Coifman and Weiss showed in \cite{CW} that $H^1(\mathbb R^n)$ is the dual space of $VMO(\mathbb R^n)$, which gives to $H^1(\mathbb R^n)$ a richer structure than $L^1(\mathbb R^n)$. For example, the classical Riesz transforms $\nabla (-\Delta)^{-1/2}$ are not bounded on $L^1(\mathbb R^n)$, but are bounded on  $H^1(\mathbb R^n)$. In addition, the weak$^*$-convergence is true in $H^1(\mathbb R^n)$ (see \cite{JJ}), which is useful  in the application of Hardy spaces to compensated compactness (see \cite{CLMS}) and in the study of commutators of singular integral operators (see \cite{Ky1, Ky3}). Let $L= -\Delta+ V$ be a Schr\"odinger operator on $\mathbb R^n$, $n\geq 3$, where $V$ is a nonnegative function, $V\ne 0$, and belongs to the reverse H\"older class $RH_{n/2}(\R^n)$. The Hardy space associated with the Schr\"odinger operator $L$, $H^1_L(\R^n)$, is then defined as the set of functions $f\in L^1(\mathbb R^n)$ such that $\|f\|_{H^1_L}:=\|\mathcal M_Lf\|_{L^1}<\infty$, where $\mathcal M_L f(x): = \sup_{t>0}|e^{-tL}f(x)|$. Recently, Ky \cite{Ky2} established that the weak$^*$-convergence is true in $H^1_L(\mathbb R^n)$, which is useful in studying the endpoint estimates for commutators of singular integral operators related to $L$ (see \cite{Ky3}).

Let $(\mathcal X, d, \mu)$ be an RD-space, which means that $(\mathcal X, d, \mu)$ is a space of homogeneous type in the sense of Coifman-Weiss with the additional property that a reverse doubling property holds in $\mathcal X$ (see Section 2). Typical examples for such RD-spaces include Euclidean spaces, Heisenberg groups, Lie groups of polynomial growth, or more generally, Carnot-Carath\'eodory spaces with doubling measures. We refer to the seminal paper of Han, M\"uller and Yang \cite{HMY2} for a systematic study of the theory of function spaces in harmonic analysis on RD-spaces. Recently, Yang and Zhou \cite{YZ} introduced and studied the theory of localized Hardy spaces $H^1_\rho(\X)$ related to the admissible functions $\rho$. There, they showed that this theory has a wide range of applications in studying the theory of Hardy spaces associated with Schr\"odinger operators or degenerate Schr\"odinger operators on $\R^n$, or associated with sub-Laplace Schr\"odinger operators on Heisenberg groups or connected and simply connected nilpotent Lie groups, see \cite[Section 5]{YZ} for details.

Given a complete RD-space $(\mathcal X, d, \mu)$ and an admissible function $\rho$, we denote by $BMO_\rho(\X)$ the dual space of $H^1_\rho(\X)$ (see Section 2) and $VMO_\rho(\X)$ the closure in the $BMO_\rho$-norm of the space $C_c(\X)$ of all continuous functions with compact support. The aim of the present paper is to show that $H^1_\rho(\X)$ is a dual space and that the weak$^*$-convergence is true in $H^1_\rho(\X)$. Our main results can be read as follows:

\begin{theorem}\label{the first main theorem}
The space $H^1_\rho(\X)$ is the dual of the space $VMO_\rho(\X)$.
\end{theorem}

\begin{theorem}\label{the second main theorem}
Suppose that $\{f_j\}_{j\geq 1}$ is a bounded sequence in $H^1_\rho(\X)$, and that $f_j(x)\to f(x)$ for almost every $x\in\X$. Then, $f\in H^1_\rho(\X)$ and $\{f_j\}_{j\geq 1}$ weak$^*$-converges to $f$, that is, for every $\varphi\in VMO_\rho(\X)$, we have
$$\lim_{j\to\infty} \int_{\X} f_j(x) \varphi(x)d\mu(x) = \int_{\X} f(x) \varphi(x) d\mu(x).$$
\end{theorem}

It should be pointed out that when $\X\equiv \R^n$, $n\geq 3$, and $\rho(x)\equiv \sup\{r>0: \frac{1}{r^{n-2}}\int_{B(x,r)}V(y)dy\leq 1\}$, where $V$ is in the reverse H\"older class $RH_{n/2}(\R^n)$, Theorem \ref{the second main theorem} is just the main theorem in the paper of Ky \cite[Theorem 1.1]{Ky2}.

Throughout the whole paper, $C$ denotes a positive geometric constant which is independent of the main parameters, but may change from line to line.

%%%%%%%%%%%%%%%%%%%%%%%%%%%%%%%%%%%%%%%%%%%%%%%%%%%%%%%%%%%%%%%%%%%%%%%%

\section{Preliminaries}

Let $d$ be a quasi-metric on a set $\X$, that is, $d$ is a nonnegative function on $\mathcal X\times \mathcal X$ satisfying
\begin{enumerate}[(a)]
	\item $d(x,y)=d(y,x)$,
	\item $d(x,y)>0$ if and only if $x\ne y$,
	\item there exists a constant $\kappa\geq 1$ such that for all $x,y,z\in \mathcal X$,
$$
	d(x,z)\leq \kappa(d(x,y)+ d(y,z)).
$$
\end{enumerate}
A trip $(\mathcal X, d,\mu)$ is called a {\sl space of homogeneous type} in the sense of Coifman-Weiss if $\mu$ is a regular Borel measure satisfying {\sl doubling property}, i.e. there exists a constant $C>1$ such that for all $x\in \mathcal X$ and $r>0$,
$$\mu(B(x,2r))\leq C \mu(B(x,r)).$$

\begin{remark}\label{compactness related to a ball}
	By \cite[Theorem (3.2)]{CW}, we see that if $(\mathcal X, d,\mu)$ is a complete space of homogeneous type, then the closure of $B$ is a compact set for all ball $B\subset \X$.
\end{remark}

Recall (see \cite{HMY2}) that a space of homogeneous type $(\mathcal X, d,\mu)$ is called an {\sl RD-space} if it satisfies {\sl reverse doubling property}, i.e. there exists a constant $C>1$ such that 
$$\mu(B(x,2r))\geq C \mu(B(x,r))$$
for all $x\in \mathcal X$ and $r\in (0, \mbox{diam}(\mathcal X)/2)$, where $\mbox{diam}(\mathcal X) := \sup_{x,y\in\mathcal X} d(x,y)$.

Here and what in follows, for $x, y\in\X$ and $r>0$, we denote $V_r(x):= \mu(B(x,r))$ and $V(x,y):= \mu(B(x,d(x,y)))$.

\begin{definition}\label{definition for test functions}
	Let $x_0\in\mathcal X$, $r>0$, $0<\beta\leq 1$ and $\gamma >0$. A function $f$ is said to belong to the space of test functions, $\mathcal G(x_0,r,\beta,\gamma)$, if there exists a positive constant $C_f$ such that
	\begin{enumerate}[\rm (i)]
		\item $|f(x)| \leq C_f \frac{1}{V_r(x_0) + V(x_0,x)}\Big(\frac{r}{r+ d(x_0,x)}\Big)^\gamma$ for all $x\in\mathcal X$;
		
		\item $|f(x) - f(y)|\leq C_f  \Big(\frac{d(x,y)}{r+ d(x_0,x)}\Big)^\beta \frac{1}{V_r(x_0) + V(x_0,x)}\Big(\frac{r}{r+ d(x_0,x)}\Big)^\gamma$ for all $x,y\in \mathcal X$ satisfying that $d(x,y)\leq \frac{r + d(x_0,x)}{2\kappa}$.
	\end{enumerate}
	For any $f\in \mathcal G(x_0,r,\beta,\gamma)$, we define 
	$$\|f\|_{\mathcal G(x_0,r,\beta,\gamma)}:= \inf \{C_f: \mbox{{\rm (i)} and {\rm (ii)} hold}\}.$$
\end{definition}

Let $\rho$ be a positive function on $\X$. Following Yang and Zhou \cite{YZ}, the function $\rho$ is said to {\sl be admissible} if there exist positive constants $C_0$ and $k_0$ such that, for any $x,y\in \X$,
$$\rho(y)\leq C_0 [\rho(x)]^{1/(1+k_0)} [\rho(x)+d(x,y)]^{k_0/(1+k_0)}.$$

{\sl Throughout the whole paper}, we always assume that $\mathcal X$ is a complete RD-space with $\mu(\mathcal X)=\infty$, and $\rho$ is an admissible function on $\X$. Also we fix $x_0\in \X$.

In Definition \ref{definition for test functions}, it is easy to see that $\mathcal G(x_0,1,\beta,\gamma)$ is a Banach space.  For simplicity, we write $\mathcal G(\beta,\gamma)$ instead of $\mathcal G(x_0,1,\beta,\gamma)$. Let $\epsilon\in (0,1]$ and $\beta,\gamma\in (0,\epsilon]$, we define the space $\mathcal G^\epsilon_0(\beta,\gamma)$ to be the completion of $\mathcal G(\epsilon,\epsilon)$ in $\mathcal G(\beta,\gamma)$, and denote by $(\mathcal G^\epsilon_0(\beta,\gamma))'$ the space of all continuous linear functionals on $\mathcal G^\epsilon_0(\beta,\gamma)$. We say that $f$ is a {\sl distribution} if $f$ belongs to $(\mathcal G^\epsilon_0(\beta,\gamma))'$.

Remark that, for any  $x\in \mathcal X$ and $r>0$, one has $\mathcal G(x,r,\beta,\gamma)= \mathcal G(x_0,1,\beta,\gamma)$  with equivalent norms, but of course the constants are depending on $x$ and $r$.

Let $f$ be a distribution in $(\mathcal G^\epsilon_0(\beta,\gamma))'$. We define {\sl the  grand maximal functions} $\M(f)$ and  $\M_\rho(f)$ as following
$$\M(f)(x) := \sup\{|\langle f,\varphi \rangle|: \varphi\in \mathcal G^\epsilon_0(\beta,\gamma), \|\varphi\|_{\mathcal G(x,r,\beta,\gamma)}\leq 1\; \mbox{for some}\; r>0\},$$
$$\M_\rho(f)(x) := \sup\{|\langle f,\varphi \rangle|: \varphi\in \mathcal G^\epsilon_0(\beta,\gamma), \|\varphi\|_{\mathcal G(x,r,\beta,\gamma)}\leq 1\; \mbox{for some}\; r\in (0,\rho(x))\}.$$

\begin{definition}
	Let $\epsilon\in (0,1)$ and $\beta,\gamma\in (0,\epsilon)$.
	\begin{enumerate}[\rm (i)]
		\item The Hardy space $H^1(\mathcal X)$ is defined by
		$$H^1(\mathcal X) = \{f\in (\mathcal G^\epsilon_0(\beta,\gamma))': \|f\|_{H^1}:= \|\M( f)\|_{L^1}<\infty \}.$$
		\item The Hardy space $H^1_\rho(\mathcal X)$ is defined by
		$$H^1_\rho(\mathcal X) = \{f\in (\mathcal G^\epsilon_0(\beta,\gamma))': \|f\|_{H^1_\rho}:= \|\M_\rho( f)\|_{L^1}<\infty \}.$$
			\end{enumerate}
\end{definition}

\begin{remark}
It was established in \cite{GLY} that the space $H^1(\X)$ coincides with the atomic Hardy $H^1_{\rm at}(\X)$ of Coifman and Weiss \cite{CW}. Moreover, for all $f\in H^1(X)$,
$$\|f\|_{L^1}\leq C \|f\|_{H^1_\rho}\leq C \|f\|_{H^1}.$$
\end{remark}

Recall (see \cite{CW}) that a function $f\in L^1_{\rm loc}(\X)$ is said to be in $BMO(\X)$ if
$$\|f\|_{BMO}:=\sup_{B} \frac{1}{\mu(B)}\int_{B}\Big|f(x) - \frac{1}{\mu(B)}\int_B f(y)d\mu(y)\Big| d\mu(x)<\infty,$$
where the supremum is taken all over balls $B\subset\X$. Denote by $VMO(\X)$ the closure in $BMO$ norm of $C_c(\X)$. The following is well-known (see \cite{CW}).

\begin{theorem}\label{Coifman-Weiss}
	\begin{enumerate}[\rm (i)]
		\item The space $BMO(\X)$ is the dual space of $H^1(\X)$.
		\item The space $H^1(\X)$ is the dual space of $VMO(\X)$.
	\end{enumerate}
\end{theorem}

\begin{definition}
	Let $\rho$ be an admissible function and $\mathcal D:= \{B(x,r)\subset \X: r\geq \rho(x)\}$. A function $f\in L^1_{\rm loc}(\X)$ is said to be in $BMO_\rho(\X)$ if
	$$\|f\|_{BMO_\rho}:= \|f\|_{BMO} + \sup_{B\in \mathcal D}\frac{1}{\mu(B)}\int_B |f(x)| d\mu(x) <\infty.$$
\end{definition}

It was established in \cite{YYZ} that

\begin{theorem}
The space $BMO_\rho(\X)$ is the dual space of $H^1_\rho(\X)$.	
\end{theorem}

%%%%%%%%%%%%%%%%%%%%%%%%%%%%%%%%%%%%%%%%%%%%%%%%%%%%%%%%%%%%%%%%%%%%%%%%

\section{Proof of Theorems \ref{the first main theorem} and \ref{the second main theorem}}

We begin by recalling the following (see \cite[Proposition 3.1]{YZ}).

\begin{lemma}\label{Yang and Zhou, Proposition 3.1}
	Let $\rho$ be an admissible function. Then, there exists a function $K_\rho: \X\times \X\to \R$ and a positive constant $C$ such that
	\begin{enumerate}[\rm (i)]
		\item $K_\rho(x,y)\geq 0$ for all $x,y\in \X$, and $K_\rho(x,y)=0$ if $d(x,y)>C \min\{\rho(x),\rho(y)\}$;
		\item $K_\rho(x,y) \leq C \frac{1}{\mu(B(x,\rho(x)))+ \mu(B(y,\rho(y)))}$ for all $x,y\in \X$;
		\item $K_\rho(x,y)= K_\rho(y,x)$ for all $x,y\in \X$;		
		\item $|K_\rho(x,y) - K_\rho(x,y')|\leq C \frac{d(y,y')}{\rho(x)} \frac{1}{\mu(B(x,\rho(x)))+ \mu(B(y,\rho(y)))}$ for all $x,y,z\in \X$ with $d(y,y') \leq [\rho(x) + d(x,y)]/2$;		
		\item  for any $x,x',y,y'\in \X$ satisfying $d(x,x') \leq [\rho(y) + d(x,y)]/3$ and $d(y,y') \leq [\rho(x) + d(x,y)]/3$, we have
		\begin{eqnarray*}
			\Big|[K_\rho(x,y) - K_\rho(x,y')] &-& [K_\rho(x',y) - K_\rho(x',y')]\Big|\\
			&\leq& C \frac{d(x,x')}{\rho(y)} \frac{d(y,y')}{\rho(x)} \frac{1}{\mu(B(x,\rho(x)))+ \mu(B(y,\rho(y)))};
		\end{eqnarray*}		
		\item $\int_{\X} K_\rho(x,y) d\mu(x)=1$ for all $y\in \X$.
	\end{enumerate}
\end{lemma}

Given a function $f$ in $L^1(\X)$, following \cite{YZ}, we define
$$K_\rho(f)(x)= \int_{\X} K_\rho(x,y) f(y) d\mu(y)$$
for all $x\in \X$. It follows directly from Lemma \ref{Yang and Zhou, Proposition 3.1} that
\begin{equation}\label{symmetric of the function K}
\int_{\X} K_\rho(f)(x) g(x) d\mu(x)= \int_{\X} K_\rho(g)(x) f(x) d\mu(x)
\end{equation}
for all $f\in L^1(\X)$ and $g\in L^\infty(\X)$. Moreover, by Remark \ref{compactness related to a ball}, 
\begin{equation}\label{continuous with compact support}
K_\rho(\phi)\in C_c(\X)\;\;\mbox{for all $\phi\in C_c(\X)$},
\end{equation}
and, for any $x\in \X$, the function $\mathbb K_\rho(x,\cdot): \X\to \R$, defined by
\begin{equation}\label{convolution of kernels}
\mathbb K_\rho(x,z):= \int_{\X} K_{\rho}(x,y)K_\rho(y,z)d\mu(y),
\end{equation}
is in $C_c(\X)$. Remark that $K_\rho(K_\rho(f))(x)= \int_{\X} \mathbb K_\rho(x,z)f(z) d\mu(z)$.

The following lemma is due to Yang and Zhou \cite{YZ}.

\begin{lemma}\label{Yang and Zhou, the key lemma}
	There exists a positive constant $C$ such that
	\begin{enumerate}[\rm (i)]
		\item for any $f\in L^1(\X)$,
		$$\|K_\rho(f)\|_{H^1_\rho} \leq C \|f\|_{L^1};$$
		\item for any $g\in H^1_\rho(\X)$,
		$$\|g- K_\rho(g)\|_{H^1} \leq C \|g\|_{H^1_\rho}.$$	
	\end{enumerate}
\end{lemma}

As a consequence of Lemma \ref{Yang and Zhou, the key lemma} and (\ref{symmetric of the function K}), for any $\phi\in C_c(\X)$,
\begin{equation}\label{from big bmo to small bmo}
\|\phi - K_\rho(K_\rho(\phi))\|_{BMO_\rho}\leq C \|\phi\|_{BMO}.
\end{equation}

Now we are ready to give the proofs of Theorems \ref{the first main theorem} and  \ref{the second main theorem}.

\begin{proof}[Proof of Theorem \ref{the first main theorem}]
	Since $VMO_\rho(\X)$ is a subspace of  $BMO_\rho(\X)$, which is the dual space of $H^1_\rho(\X)$, every function $f$ in $H^1_\rho(\X)$ determines a bounded linear functional on $VMO_\rho(\X)$ of norm bounded by $\|f\|_{H^1_\rho}$.
	
	Conversely, given a bounded linear functional $L$ on $VMO_\rho(\X)$. Then, 
	$$|L(\phi)|\leq \|L\| \|\phi\|_{VMO_\rho}\leq \|L\| \|\phi\|_{L^\infty}$$
	for all $\phi\in C_c(\X)$. This implies (see \cite{Ro}) that there exists a finite signed Radon measure $\nu$ on $\X$ such that, for any $\phi\in C_c(\X)$, 
	$$L(\phi)= \int_{\X} \phi(x) d\nu(x),$$
	moreover,  the total variation of $\nu$, $|\nu|(\X)$,  is bounded by $\|L\|$. Therefore, 
	\begin{equation}\label{Dafni 1, 1}
	\|K_{\rho}(K_\rho(\nu))\|_{H^1_\rho} \leq C \|K_\rho(\nu)\|_{L^1}\leq C |\nu|(\X)\leq C \|L\|
	\end{equation}
	by Lemma \ref{Yang and Zhou, the key lemma}, where $K_\rho(\nu)(x):= \int_{\X} K_\rho(x,y) d\nu(y)$ for all $x\in \X$.

	 On the other hand, by (\ref{from big bmo to small bmo}) and (\ref{continuous with compact support}), we have
	\begin{eqnarray*}
		|(L- K_\rho(K_\rho(L)))(\phi)| &=&|L(\phi - K_\rho(K_\rho(\phi)))|\\
		&\leq& \|L\| \|\phi - K_\rho(K_\rho(\phi))\|_{VMO_\rho}\\
		&\leq& C \|L\| \|\phi\|_{BMO}
	\end{eqnarray*}
	for all $\phi\in C_c(\X)$, where $K_\rho(K_\rho(L))(\phi):= \int_{\X} K_\rho(K_\rho(\nu))(x) \phi(x)d\mu(x)$. Consequently, by Theorem \ref{Coifman-Weiss}(ii), there exists a function $h$ belongs $H^1(\X)$ such that $\|h\|_{H^1}\leq C \|L\|$ and
	$$(L- K_\rho(K_\rho(L)))(\phi)= \int_{\X} h(x) \phi(x) d\mu(x)$$
	for all  $\phi\in C_c(\X)$. This, together with (\ref{Dafni 1, 1}), allows us to conclude that
	$$L(\phi)= \int_{\R^d} f(x) \phi(x) d\mu(x)$$
	for all $\phi\in C_c(\X)$, where $f:= h+ K_{\rho}(K_\rho(\nu))$ is in $H^1_\rho(\X)$ and satisfies that $\|f\|_{H^1_\rho}\leq \|h\|_{H^1_\rho} + \|K_{\rho}(K_\rho(\nu))\|_{H^1_\rho}\leq C \|L\|$. The proof of Theorem \ref{the first main theorem} is thus completed.
\end{proof}

\begin{proof}[Proof of Theorem \ref{the second main theorem}]
	Let $\{f_{n_k}\}_{k=1}^\infty$ be an arbitrary subsequence of $\{f_n\}_{n=1}^\infty$.	As $\{f_{n_k}\}_{k=1}^\infty$ is a bounded sequence in $H^1_\rho(\X)$, by Theorem \ref{the first main theorem} and the Banach-Alaoglu theorem, there exists a subsequence $\{f_{n_{k_j}}\}_{j=1}^\infty$ of $\{f_{n_k}\}_{k=1}^\infty$ such that $\{f_{n_{k_j}}\}_{j=1}^\infty$ weak$^*$-converges to $g$ for some $g\in H^1_\rho(\X)$. Therefore, by (\ref{convolution of kernels}), for any  $x\in \X$,
	\begin{eqnarray*}	
	\lim_{j\to\infty} K_\rho(K_\rho(f_{n_{k_j}}))(x) &=& \lim_{j\to\infty} \int_{\X} \mathbb K_\rho(x,z) f_{n_{k_j}}(z) d\mu(z)\\
	&=& \int_{\X} \mathbb K_\rho(x,z) g(z) d\mu(z)	 = K_\rho(K_\rho(g))(x).
	\end{eqnarray*}
	This implies that $\lim_{j\to\infty}[f_{n_{k_j}}(x) - K_\rho(K_\rho(f_{n_{k_j}}))(x)]= f(x) - K_\rho(K_\rho(g))(x)$ for almost every $x\in\X$. Hence, by Lemma \ref{Yang and Zhou, the key lemma} and \cite[Theorem 1.1]{HK},  
	$$\|f- K_\rho(K_\rho(g))\|_{H^1}\leq \sup_{j\geq 1}\|f_{n_{k_j}} - K_\rho(K_\rho(f_{n_{k_j}})\|_{H^1}\leq C \sup_{j\geq 1} \|f_{n_{k_j}}\|_{H^1_\rho}<\infty,$$
	moreover,
	$$\lim_{j\to\infty} \int_{\X} [f_{n_{k_j}}(x) - K_\rho(K_\rho(f_{n_{k_j}}))(x)]\phi(x)d\mu(x) =  \int_{\X}  [f(x) - K_\rho(K_\rho(g))(x)]\phi(x) d\mu(x)$$
	for all $\phi\in C_c(\X)$. As a consequence, we obtain that
	\begin{eqnarray*}
		\|f\|_{H^1_\rho} &\leq& \|f- K_\rho(K_\rho(g))\|_{H^1_\rho} + \|K_\rho(K_\rho(g))\|_{H^1_\rho}\\
		&\leq& C\|f- K_\rho(K_\rho(g))\|_{H^1} + C\|g\|_{H^1_\rho}\\
		&\leq& C \sup_{j\geq 1} \|f_{n_{k_j}}\|_{H^1_\rho}<\infty,
	\end{eqnarray*}
	moreover, by $\{f_{n_{k_j}}\}_{j=1}^\infty$ weak$^*$-converges to $g$ in $H^1_\rho(\X)$, (\ref{symmetric of the function K}) and (\ref{continuous with compact support}),
	\begin{eqnarray*}
		&&\lim_{j\to\infty} \int_{\X} f_{n_{k_j}}(x)\phi(x)d\mu(x) \\
		&=& \lim_{j\to\infty} \int_{\X} [f_{n_{k_j}}(x) - K_\rho(K_\rho(f_{n_{k_j}}))(x)]\phi(x)d\mu(x) + \lim_{j\to\infty} \int_{\X} f_{n_{k_j}}(x) K_\rho(K_\rho(\phi))(x)d\mu(x)\\
		&=& \int_{\X}  [f(x) - K_\rho(K_\rho(g))(x)]\phi(x) d\mu(x) + \int_{\X} g(x)  K_\rho(K_\rho(\phi))(x)d\mu(x)\\
		&=& \int_{\X} f(x)\phi(x)d\mu(x).
	\end{eqnarray*}
	This, by $\{f_{n_k}\}_{k=1}^\infty$ be an arbitrary subsequence of $\{f_n\}_{n=1}^\infty$,  allows us to complete the proof of Theorem \ref{the second main theorem}.
\end{proof}

\section{An application}

The purpose of this section is to give an atomic characterization of  $H^1_\rho(\X)$ by using Theorems \ref{the first main theorem} and \ref{the second main theorem}. First, we define the concept of atoms of log-type.

\begin{definition}
	Given $1<q\leq \infty$. A measurable function $a$ is called an $(H^1_\rho,q)$-atom of log-type related to the ball $ B(x_0,r)$ if 
	\begin{enumerate}[\rm (i)]
		\item  supp $a\subset B(x_0,r)$,
		
		\item $\|a\|_{L^q(\X)}\leq [\mu(B(x_0,r))]^{1/q-1}$,
		
		\item $\left|\int_{\X}a(x)d\mu(x)\right| \leq \frac{1}{\log\left(e+ \frac{\rho(x_0)}{r}\right)}$.
	\end{enumerate}
\end{definition}

The main result in this section can be read as follows:

\begin{theorem}\label{the atomic characterization}
	Let $1<q\leq \infty$. A function $f$ is in $H^1_\rho(\X)$ if and only if it can be written as $f=\sum_j \lambda_j a_j$, where $a_j$ are $(H^1_\rho,q)$-atoms of log-type and $\sum_j |\lambda_j|<\infty$. Moreover, there exists a constant $C>0$ such that, for any $f\in H^1_\rho(\X)$,
	$$\|f\|_{H^1_\rho}\leq C \inf\left\{\sum_j |\lambda_j|: f=\sum_j \lambda_j a_j\right\}\leq C \|f\|_{H^1_\rho}.$$
\end{theorem}

Before giving the proof of Theorem \ref{the atomic characterization}, let us recall the definition of $H^1_\rho$ atoms introduced by Yang and Zhou \cite{YZ}.

\begin{definition}
	Given $1<q\leq \infty$. A measurable function $a$ is called an $(H^1_\rho,q)$-atom related to the ball $ B(x_0,r)$ if $r< \rho(x_0) $ and
	\begin{enumerate}[\rm (i)]
\item  supp $a\subset B(x_0,r)$,
	
\item $\|a\|_{L^q(\X)}\leq [\mu(B(x_0,r))]^{1/q-1}$,
	
\item if $r< \rho(x_0)/4$, then $\int_{\X}a(x)d\mu(x)=0$.
	\end{enumerate}
\end{definition}

\begin{remark}\label{atom implies atom of log-type}
If $a$ is an $(H^1_\rho,q)$-atom, then $\frac{1}{\log(e+4)}a$ is an $(H^1_\rho,q)$-atom of log-type.
\end{remark}

%The following atomic characterization of $H^1_\rho(\X)$ was established in \cite{YZ}.
%
%\begin{theorem}
%	Let $1<q\leq \infty$. A function $f$ is in $H^1_\rho(\X)$ if and only if it can be written as $f=\sum_j \lambda_j a_j$, where $a_j$ are $(H^1_\rho,q)$-atoms and $\sum_j |\lambda_j|<\infty$. Moreover, there exists a constant $C>0$ such that, for any $f\in H^1_\rho(X)$,
%	$$\|f\|_{H^1_\rho}\leq \inf\left\{\sum_j |\lambda_j|: f=\sum_j \lambda_j a_j\right\}\leq C \|f\|_{H^1_\rho}.$$
%\end{theorem}

\begin{proof}[Proof of Theorem \ref{the atomic characterization}]
By Remark \ref{atom implies atom of log-type} and \cite[Theorems 3.2]{YZ}, it suffices to prove that there exists a constant $C>0$ such that if $f$ can be written as $f=\sum_j \lambda_j a_j$, where $a_j$ are $(H^1_\rho,q)$-atoms of log-type related to the balls $B(x_j, r_j)$ and $\sum_j |\lambda_j|<\infty$, then $\|f\|_{H^1_\rho} \leq C \sum_{j} |\lambda_j|$. Since Theorem  \ref{the second main theorem}, we only need to prove that
$$
\|a_j\|_{H^1_\rho}\leq C
$$
for all $j$. This is reduced to showing that, for any $\phi\in C_c(\X)$,
\begin{equation}\label{the atomic characterization, 1}
\left|\int_{\X} a_j(x) \phi(x) d\mu(x) \right|\leq C \|\phi\|_{BMO_\rho}
\end{equation}
by Theorem \ref{the first main theorem}. To prove (\ref{the atomic characterization, 1}), let us consider the following two cases:
\begin{enumerate}[(a)]
	\item The case: $r_j\geq \rho(x_j)$. Then, by the H\"older inequality and \cite[Lemma 2.2]{YZ},
	\begin{eqnarray*}
	\left|\int_{\X} a_j(x) \phi(x) d\mu(x) \right| &\leq& \|a_j\|_{L^q(B(x_j,r_j))} \|\phi\|_{L^{q'}(B(x_j,r_j))}\\
	&\leq& [\mu(B(x_j,r_j))]^{1/q-1} C [\mu(B(x_j,r_j))]^{1/q'} \|\phi\|_{BMO_\rho}\\
	&\leq& C \|\phi\|_{BMO_\rho}
	\end{eqnarray*}
	where and hereafter $1/q'+1/q=1$.
	\item The case: $r_j< \rho(x_j)$. Then, by the H\"older inequality, \cite[Lemma 2.2]{YZ} and \cite[Lemma 2.1]{LY},
	\begin{eqnarray*}
	\left|\int_{\X} a_j(x) \phi(x) d\mu(x) \right|	 &\leq& \left|\int_{\X} a_j(\phi- \phi_{B(x_j,r_j)}) d\mu \right| + |\phi_{B(x_j,r_j)}| \left|\int_{\X} a_jd\mu \right|\\
	&\leq&  \|a_j\|_{L^q(B(x_j,r_j))} \|\phi- \phi_{B(x_j,r_j)}\|_{L^{q'}(B(x_j,r_j))} + C \|\phi\|_{BMO_\rho}\\
	&\leq& [\mu(B(x_j,r_j))]^{1/q-1} C [\mu(B(x_j,r_j))]^{1/q'} \|\phi\|_{BMO_\rho}+ C \|\phi\|_{BMO_\rho}\\
	&\leq& C \|\phi\|_{BMO_\rho},
	\end{eqnarray*}
	where $\phi_{B(x_j,r_j)}:= \frac{1}{\mu(B(x_j,r_j))}\int_{B(x_j,r_j)} \phi d\mu$. This ends the proof of Theorem \ref{the atomic characterization}.
\end{enumerate}

\end{proof}

{\bf Acknowledgements.} The paper was completed when the third author was visiting
to Vietnam Institute for Advanced Study in Mathematics (VIASM). He would like
to thank the VIASM for financial support and hospitality.

\bigskip

\noindent Dinh Thanh Duc

\medskip

\noindent Department of Mathematics,
University of Quy Nhon,
170 An Duong Vuong,
Quy Nhon, Binh Dinh, Vietnam
\smallskip

\noindent {\it E-mail}: \texttt{dinhthanhduc@qnu.edu.vn}

\bigskip

\noindent Ha Duy Hung

\medskip

\noindent High School for Gifted Students,
Hanoi National University of Education, 136 Xuan Thuy, Hanoi, Vietnam

\smallskip

\noindent {\it E-mail:} \texttt{hunghaduy@gmail.com}

\bigskip

\noindent Luong Dang Ky

\medskip

\noindent Department of Mathematics,
University of Quy Nhon,
170 An Duong Vuong,
Quy Nhon, Binh Dinh, Vietnam
\smallskip

\noindent {\it E-mail}: \texttt{luongdangky@qnu.edu.vn}


\begin{thebibliography}{MTW1}
\bibitem{CLMS} R. R. Coifman, P.-L. Lions, Y. Meyer and S. Semmes, Compensated compactness and Hardy spaces. J. Math. Pures Appl. (9) 72 (1993), no. 3, 247--286.

\bibitem{CW} R. R. Coifman and G. Weiss,  Extensions of Hardy spaces and their use in analysis. Bull. Amer. Math. Soc. 83 (1977), no. 4, 569-645.

\bibitem{GLY} L. Grafakos, L. Liu and D. Yang,  Maximal function characterizations of Hardy spaces on RD-spaces and their applications. Sci. China Ser. A 51 (2008), no. 12, 2253--2284. 

\bibitem{HMY2} Y. Han, D. M\"uller and D. Yang,  A  theory  of  Besov  and  Triebel–Lizorkin  spaces  on  metric  measure  spaces  modeled  on Carnot-Carath\'eodory spaces, Abstr. Appl. Anal. (2008), art. ID 893409, 250 pp.

\bibitem{HK} H. D. Hung and L. D. Ky, On weak$^*$-convergence in the Hardy space $H^1$ over spaces of homogeneous type, arXiv:1510.01019.

\bibitem{HHK} H. D. Hung, D. Q. Huy and L. D. Ky, A note on weak$^*$-convergence in $h^1(\R^d)$,  arXiv:1510.03597.

\bibitem{JJ} P. W. Jones and J-L. Journ\'e, On weak convergence in $H^1({\bf R}^d)$.
Proc. Amer. Math. Soc. 120 (1994), no. 1, 137-138. 

\bibitem{Ky1} L. D. Ky, Bilinear decompositions and commutators of singular integral operators.  Trans. Amer. Math. Soc. 365 (2013), no. 6, 2931--2958.

\bibitem{Ky2} L. D. Ky, On weak$^*$-convergence in $H^1_L(\mathbb R^d)$.  Potential Anal. 39 (2013), no. 4, 355--368.

\bibitem{Ky3} L. D. Ky, Endpoint estimates for commutators of singular integrals related to Schr\"odinger operators, Rev. Mat. Iberoam. (to appear) or arXiv:1203.6335.

\bibitem{LY} H. Lin and D. Yang,  Pointwise multipliers for localized Morrey-Campanato spaces on RD-spaces. Acta Math. Sci. Ser. B Engl. Ed. 34 (2014), no. 6, 1677--1694.



%\bibitem{MS}  R. A. Mac\'ias and C. Segovia,  Lipschitz functions on spaces of homogeneous type. Adv. in Math. 33 (1979), no. 3, 257-270.

\bibitem{Ro} H. L. Royden, Real analysis. Third edition. Macmillan Publishing Company, New York, 1988.

\bibitem{YYZ}  D. Yang, D. Yang and Y. Zhou, Localized Morrey-Campanato spaces on metric measure spaces and applications to Schr\"odinger operators, Nagoya Math. J. 198 (2010), 77-119.


\bibitem{YZ} D. Yang and Y. Zhou, Localized Hardy spaces $H^1$ related to admissible functions on RD-spaces and applications to Schr\"odinger operators. Trans. Amer. Math. Soc. 363 (2011), no. 3, 1197-1239.



\end{thebibliography}
\end{document}